\documentclass[a4paper,10pt,oneside,reqno]{amsart}
\usepackage{amssymb,amscd}
\input{xy}
\xyoption{all}

\newtheorem{theo}{Theorem}[section]

\theoremstyle{definition}

\theoremstyle{remark}
\newtheorem{rem}[theo]{Remark}
\newtheorem*{rem*}{Remark}
\newtheorem{que}[theo]{Question}

\let\ge\geqslant
\let\le\leqslant
\let\emptyset\varnothing
\let\ES\varnothing

\let\phi\varphi
\let\kappa\varkappa
\let\hra\hookrightarrow

\newcommand{\uni}[1]{\mathbf{1}_{{#1}}}
\newcommand{\Comp}{\mathcal{C}\mathrm{omp}}

\newcommand{\Convmm}{\mathcal{C}\mathrm{onv}_{\scriptstyle\max,\min}}
\newcommand{\BConvmm}{\mathcal{B}\mathrm{i}\mathcal{C}\mathrm{onv}_{\scriptstyle\max,\min}}

\newcommand\subcl{\mathrel{\underset{\mathrm{cl}}{\subset}}}
\newcommand\subop{\mathrel{\underset{\mathrm{op}}{\subset}}}

\newcommand{\Mup}{\mathbb{M}_{\scriptscriptstyle\cap}}
\newcommand{\Mdn}{\mathbb{M}_{\scriptscriptstyle\cup}}
\newcommand{\mup}{M_{\scriptscriptstyle\cap}}
\newcommand{\mdn}{M_{\scriptscriptstyle\cup}}

\newcommand{\etadn}{\eta_{\scriptscriptstyle\cup}}
\newcommand{\mudn}{\mu_{\scriptscriptstyle\cup}}
\newcommand{\xidn}{\xi_{\scriptscriptstyle\cup}}
\newcommand{\xiup}{\xi_{\scriptscriptstyle\cap}}

\newcommand{\Delom}{\Delta_{\scriptscriptstyle\otimes}}
\newcommand{\Delop}{\Delta_{\scriptscriptstyle\oplus}}
\newcommand{\op}{\oplus}
\newcommand{\om}{\otimes}
\newcommand{\bp}{\mathop{\bar\oplus}}
\newcommand{\bm}{\mathop{\bar\otimes}}

\newcommand{\BBN}{\mathbb{N}}

\newcommand{\BBR}{\mathbb{R}}

\newcommand{\BBG}{\mathbb{G}}
\newcommand{\BBM}{\mathbb{M}}

\newcommand{\CCA}{\mathcal{A}}

\newcommand{\CCC}{\mathcal{C}}

\newcommand{\CCF}{\mathcal{F}}

\newcommand{\CCS}{\mathcal{S}}

\begin{document}
\title[Idempotent convexity and algebras for the capacity monad\dots]%
{Idempotent convexity and algebras for the capacity monad and its
submonads}
\author{Oleh Nykyforchyn, Du\v san Repov\v s}

\address{
Department of Mathematics and Computer Science \newline
\indent Vasyl' Stefanyk Precarpathian National University \newline
\indent Shevchenka 57, Ivano-Frankivsk, Ukraine\newline
\email{oleh.nyk@gmail.com}\newline
\mbox{}
 Faculty of Mathematics and Physics\newline
\indent and Faculty of Education\newline
\indent University of Ljubljana, P.O.B.2964\newline
\indent Ljubljana, Slovenia\newline
\email{dusan.repovs@guest.arnes.si}\newline
\mbox{}}
\begin{abstract}
Idempotent analogues of convexity are introduced. It is proved
that the category of algebras for the capacity monad in the
category of compacta is isomorphic to the category of
$(\max,\min)$-idempotent biconvex compacta and their biaffine maps.
It is also shown that the category of algebras for the monad of
sup-measures ($(\max,\min)$-idempotent measures) is isomorphic to
the category of $(\max,\min)$-idempotent convex compacta and their
affine maps.
\end{abstract}
\date{\today}
\maketitle

{\let\thefootnote\empty
\footnotetext{This research was supported by the Slovenian Research
Agency grants P1-0292-0101, J1-9643-0101 and BI-UA/09-10-002, and
the State Fund of Fundamental Research of Ukraine grant 25.1/099.}
}

\section*{Introduction}

\emph{Monads} (also called \emph{triples}, \cite{BW,EilMoore})
in topological categories and algebras for these monads are closely
related to important objects of analysis and topological algebra.
\'Swirszcz~\cite{Swir74} proved that algebras and their morphisms
for the probability measure monad are precisely convex compact maps
of locally convex vector topological spaces and continuous affine
maps.

By a result of Day (cf. Theorem 3.3 of \cite{Day75}), the category
of algebras for the filter monad in the category of sets is the
category of continuous lattices and their mappings that preserve
directed joins and arbitrary meets. Due to Wyler~\cite{Wyl81}
algebras for the hyperspace monad are compact Lawson semilattices.
Zarichnyi~\cite{ZarLambda} has shown that the category of algebras
for the superextension monad is isomorphic to the category of
compacta with (fixed) almost normal $T_2$-subbase and their convex
maps. We will use a result of Radul~\cite{RadG90} who introduced
the~inclusion hyperspace triple and proved that its algebras and
their morphisms are in fact compact Lawson lattices and their
complete homomorphisms.

Unlike probability (normed additive) measures which are a
traditional object of investigation by means of categorical
topology, their non-additive analogues were paid less attention
from this point of view. Meanwhile capacities (normed non-additive
measures) that were introduced by Choquet~\cite{Ch} and
rediscovered by Sugeno under the name {\sl fuzzy measures}
have found
numerous applications, e.g. in decision making under
uncertainty~\cite{EK,Sch}. One of the most promising classes of
non-additive measures is one of idempotent measures~\cite{Ak}. For
other important classes of capacities and their topological
properties see \cite{BriVer94}. Upper semicontinuous capacities on
compact spaces were systematically studied in~\cite{Zh}.

Therefore it seems natural to use methods of categorical topology
to study non-additive measures. Nykyforchyn and
Zarichnyi~\cite{CapZN} defined the capacity functor and the
capacity monad in the category of compacta, and proved basic
topological properties of capacities on metrizable and
non-metrizable compacta. Two important dual subfunctors of the
capacity functor, namely of $\cup$-capacities (possibility
measures) and of $\cap$-capacities (necessity measures) were
introduced in \cite{NHl}, and it was shown that they lead to
submonads of the capacity monad. The aim of this paper is to
describe categories of algebras for the capacity monad, for the
monads of $\cup$-capacities and of $\cap$-capacities, and to
present internal relations of the capacity monad and its submonads
with idempotent mathematics and generalizations of convexity (in
the form of join geometry).

\section{Preliminaries}

A \emph{compactum} is a compact Hausdorff topological space. We
regard the unit segment $I=[0;1]$ as a subspace of the real line
with the natural topology. We write $A\subcl B$ (resp.~$A\subop B$)
if $A$ is a closed (resp.\ an open) subset of a space $B$. For a
set $X$ the identity mapping $X\to X$ is denoted by $\uni{X}$. For
a compactum $X$ we denote by $\exp X$ the set of all nonempty
closed subsets of $X$ with the \emph{Vietoris topology}. A base of
this topology consists of all sets of the form
$$
\langle U_1,U_2,\dots,U_n\rangle=
\{F\in\exp X\mid F\subset U_1\cup U_2\cup\dots\cup U_n, F\cap U_i\ne\ES
\text{ for all }1\le i \le n\},
$$
where $n\in\BBN$ and all $U_i\subset X$ are open. The space $\exp
X$ for a compactum $X$ is a compactum as well. A nonempty closed
subset $\CCF\subset
\exp X$ is called an
\emph{inclusion hyperspace} if for all $A,B\in\exp X$ an inclusion
$A\subset B$ and $A\in\CCF$ imply $B\in\CCF$. The set $GX$ of all
inclusion hyperspaces is closed in $\exp(\exp X)$. For more on
$\exp X$ and $GX$ see~\cite{TZ}.

We regard any set $S$ with an~idempotent, commutative and
associative binary operation $\op:S\times S\to S$ (with an additive
denotation) as an~upper semilattice with the partial order $x\le
y\iff y\ge x\iff x\op y=y$ and the pairwise supremum $x\op y$ for
$x,y\in S$. Similarly, given an~idempotent, commutative and
associative operation $\om:S\times S\to S$ (with a~multiplicative
denotation), we regard $S$ as a lower semilattice with the partial
order $x\le y\iff y\ge x\iff x\om y=x$ and $x\om y$ being the
infimum of $x,y\in S$.

If two operations $\op,\om:L\times L\to L$ are idempotent,
commutative and associative, and the distributive laws and the laws
of absorption are valid, then $L$ is a distributive lattice w.r.t.\
the partial order $x\le y\iff y\ge x\iff x\op y=y\iff x\om y=x$,
and $x\op y$ and $x\om y$ are the pairwise supremum and the
pairwise infimum of $x,y\in L$.

If $f,g$ are functions with the same domain and values in a poset,
then by $f\lor g$ and $f\land g$ we also define their poinwise
supremum and infimum. If $f$ is a function with values in a set $L$
with an operation ``$\op$'' (or ``$\om$''), and $\alpha\in L$, then
$(\alpha\op f)(x)=\alpha\op f(x)$ (resp.\ $(\alpha\om
f)(x)=\alpha\om f(x)$) for any valid argument $x$.

An~\emph{idempotent semiring} is a set $R$ with binary operations
$\op,\om:R\times R\to R$ such that $(R,\op)$ is an~abelian monoid
with a neutral element $0$, ``$\op$'' is idempotent, i.e. $a\op
a=a$ for all $a\in R$, $(R,\om)$ is a~monoid with a neutral element
$1$, the operation ``$\om$'' is distributive over ``$\op$''~:
$a\om(b\op c)=(a\om b)\op(a\om c)$ for all $a,b,c\in R$, and $0\om
a=a\om 0=0$ for all $a\in R$. The most popular idempotent semiring
is the~\emph{tropical semiring} $(\BBR\cup\{-\infty\},\op,\om)$,
where $x\op y=\max\{x,y\}$, $x\om y=x+y$, which is the basis of
\emph{tropical mathematics}~\cite{KolMas}. A~little less
extensively studied is the~idempotent semiring
$(\BBR\cup\{\pm\infty\},\op,\om)$, where $x\op y=\max\{x,y\}$,
$x\om y=\min\{x,y\}$. We will use a semiring which is algebraically
and topologically isomorphic to it, but more convenient for our
purposes, namely $(I,\op,\om)$ with $x\op y=\max\{x,y\}$, $x\om
y=\min\{x,y\}$. In general, any distributive lattice $(L,\op,\om)$
with top and bottom elements is an idempotent semiring.

See~\cite{BW,ML} for the definitions of category, morphism,
functor, natural transformation, monad, algebra for a monad,
morphism of algebras, tripleability and related facts. By
$\uni{\CCC}$ we denote the identity functor in a category $\CCC$.
Recall that all $\mathbb{F}$-algebras for a fixed monad
$\mathbb{F}$ and all their morphisms form a~\emph{category of
$\mathbb{F}$-algebras}.

It is proved in \cite{TZ} that costructions $\exp$ and $G$ can be
extended to functors in $\Comp$ that are functorial parts of
monads. For a continuous map of compacta $f:X\to Y$ the maps $\exp
f:\exp X\to
\exp Y$ and $Gf:GX\to GY$ are defined by the formulae $\exp
f(F)=\{f(x)\mid x\in F\}$, $F\in\exp X$ and $Gf(\CCF)=\{B\subcl
Y\mid B\supset f(A)\text{ for some }A\in\CCF\}$, $\CCF\in GX$. For
the \emph{inclusion hyperspace monad} $\BBG=(G,\eta_G,\mu_G)$ the
components $\eta_GX:X\to GX$ and $\mu_GX:G^2X\to GX$ of the unit
and the multiplication are defined as follows~:
$\eta_G(x)=\{F\in\exp X\mid x\in F\}$, $x\in X$ and
$\mu_GX(\mathrm{F})=\bigcup\{\bigcap\CCA\mid\CCA\in\mathrm{F}\}$,
$\mathrm{F}\in G^2X$.

We denote by $\Comp$ the
\emph{category of compacta} that consists of all compacta and their
continuous mappings. If there is a natural transformation of one
functor in $\Comp$ to another with all components being topological
embeddings, then the first functor is called a
\emph{subfunctor} of the latter~\cite{TZ}. Similarly an
\emph{embedding of monads} in $\Comp$ is a morphism of monads with
all components being topological embeddings. If there exists an
embedding of one monad in $\Comp$ into another one, then the first
monad is called a
\emph{submonad} of the latter.

Now we present the main notions and results of~\cite{CapZN,NHl}
that concern capacities on compacta, the capacity functor and the
capacity monad. We call a function $c:\exp X\cup\{\emptyset\}\to I$
a \emph{capacity} on a compactum $X$ if the following three
properties hold for all closed subsets $F$, $G$ of~$X$~:
\begin{enumerate}
\item
$c(\emptyset)=0$, $c(X)=1$;
\item
if $F\subset G$, then $c(F)\le c(G)$ (monotonicity);
\item
if $c(F)<a$, then there exists an open set $U\supset F$ such that
$G\subset U$ implies $c(G)<a$ (upper semicontinuity).
\end{enumerate}
We extend a capacity $c$ to all open subsets in $X$ by the
formula~:
$$
c(U)=\sup\{c(F)\mid F\subcl X, F\subset U\}, U\subop X.
$$

It is proved in \cite{CapZN} that the set $MX$ of all capacities on
a compactum $X$ is a compactum as well, if a topology on $MX$ is
determined by a subbase that consists of all sets of the form
$$
O_-(F,a)=\{c\in MX\mid c(F)<a\},
$$
where $F\subcl X$, $a\in \BBR$, and
\begin{multline*}
O_+(U,a)=\{c\in MX\mid c(U)>a\}=\\
\{c\in MX\mid \text{there exists a compactum } F\subset U, c(F)>a\},
\end{multline*}
where $U\subop X$, $a\in \BBR$.

The assignment $M$ extends to the \emph{capacity functor} $M$ in
the category of compacta, if the map $Mf:MX\to MY$ for a continuous
map of compacta $f:X\to Y$ is defined by the formula
$$
Mf(c)(F)=c(f^{-1}(F)),
$$
where $c\in MX$, $F\subcl Y$. This functor is the functorial part
of the \emph{capacity monad} $\mathbb{M}=(M,\eta,\mu)$ that was
described in \cite{CapZN}. Its unit and multiplication are defined
by the formulae
$$
\eta X(x)=\delta_x\text{ where }
\delta_x(F)=
\begin{cases}
1,\text{ if }x\in F,\\
0,\text{ if }x\notin F,\\
\end{cases}
\quad
\text{(a Dirac measure concentrated in $x$)}
$$
$$
\mu X(\CCC)(F)=
\sup\{\
\alpha\in I
\mid
\CCC(\{c\in MX\mid c(F)\ge\alpha\})\ge\alpha
\},
$$
where $x\in X$, $\CCC\in M^2X$, $F\subcl X$.

We call a capacity $c\in MX$ a~\emph{$\cup$-capacity} (also called
\emph{sup-measure} or \emph{possibility measure}) if
$c(A\cup B)=\max\{c(A),c(B)\}$ for all $A,B\subcl X$. A capacity
$c\in MX$ a~\emph{$\cap$-capacity} (or \emph{necessity
measure})~\cite{NHl} if $c(A\cap B)=\min\{c(A),c(B)\}$ for all
$A,B\subcl X$. The sets of all $\cup$-capacities and of all
$\cap$-capacities on a compactum $X$ are denoted by $\mdn X$ and
$\mup X$. It is proved in \cite{NHl} that $\mdn X$ and $\mup X$ are
closed in $MX$, $Mf(\mdn X)\subset \mdn Y$ and $Mf(\mup X)\subset
\mup Y$ for any continuous map of compacta $f:X\to Y$, thus we
obtain subfunctors $\mdn,\mup$ of the capacity functor~$M$.
Moreover, we get submonads $\Mdn$ and $\Mup$ of the capacity
monad~$\BBM$.

Observe that for a $\cup$-capacity $c$ and a closed set $F\subset
X$ we have $c(F)=\max\{c(x)\mid x\in F\}$, and $c$ is completely
determined by its values on singletons. Therefore we often identify
$c$ with the upper semicontinuous function $X\to I$ that sends each
$x\in X$ to $c(\{x\})$, and write $c(x)$ instead of $c(\{x\})$.
Conversely, each upper semicontinuous function $c:X\to I$ with
$\max c=1$ determines a $\cup$-capacity by the formula
$c(F)=\max\{c(x)\mid x\in F\}$, $F\subcl X$. A similar, but a
little more complicated observation is valid for $\cap$-capacities.

\section{Algebras for the monads of $\cup$-capacities and $\cap$-capacities}

Let an operation $ic:X\times I\times X\to X$ be given for a set
$X$. In the sequel we denote $ic(x,\alpha,y)$ by $x\op (\alpha\om
y)$ or simply by $x\op\alpha y$ for the sake of shortness. We call
$ic$ an~\emph{idempotent convex combination} of two points in~$X$
if the following equalities are valid for all $x,y,z\in X$,
$\alpha,\beta\in I$~:

1) $x\op \alpha x=x$;

2) $(x\op\alpha y)\op \beta z=(x\op \beta z)\op \alpha y$;

3) $x\op \alpha(y\op\beta z)=(x\op\alpha y)\op (\alpha\om\beta)z$;

4) $x\op 1y=y\op 1x$;

5) $x\op 0y=x$.

We also call the set
$$
\Delop^n=\{(\alpha_0,\alpha_1,\dots,\alpha_n)\in
I^{n+1}\mid \alpha_0\op\alpha_1\op\dots\op\alpha_n=1\}
$$
the~\emph{(idempotent) $n$-dimensional $\op$-simplex}. Now,
assuming 1)--5), for any coefficients
$(\alpha_0,\alpha_1,\dots,\alpha_n)\in\Delop^n$ and elements
$x_0,x_1,\dots,x_n\in X$ we define the
\emph{idempotent convex combination} of $n+1$ points as follows
(assume that $\alpha_k=1$ for some $0\le k\le n$)~:
\begin{multline*}
\alpha_0x_0\op\alpha_1x_n\op\dots\op\alpha_nx_n=\\
({\dots}((x_k\op\alpha_0x_0)\op\dots)\op\alpha_{k-1}x_{k-1})
\op\alpha_{k+1}x_{k+1})\op\dots)\op\alpha_nx_n.
\end{multline*}

Conditions 2),4) assure that the combination is well defined and
does not depend on the order of summands. Obviously $1x\op\alpha
y=x\op\alpha y$. By 5) summands with zero coefficients can be
dropped, and by 1) and 3), if two summands contain the same point,
then a summand with a greater coefficient absorbs a summand with a
less coefficient. Conditions 2),3) also imply a ``big associative
law''~:
\begin{multline*}
\alpha_0(\beta^0_0x^0_0\op\dots\op\beta^0_{k_0}x^0_{k_0})\op
\alpha_1(\beta^1_0x^1_0\op\dots\op\beta^1_{k_1}x^1_{k_1})\op
\dots\op
\alpha_n(\beta^n_0x^n_0\op\dots\op\beta^n_{k_n}x^n_{k_n})=
\\
(\alpha_0\om\beta^0_0)x^0_0\op\dots\op(\alpha_0\om\beta^0_{k_0})x^0_{k_0}\op
(\alpha_1\om\beta^1_0)x^1_0\op\dots\op(\alpha_1\om\beta^1_{k_1})x^1_{k_1}\op
\\
\dots\op
(\alpha_n\om\beta^n_0)x^n_0\op\dots\op(\alpha_n\om\beta^n_{k_n})x^n_{k_n},
\end{multline*}
where $x^i_j\in X$,
$(\alpha_0,\alpha_1,\dots,\alpha_n)\in\Delop^n$,
$(\beta^i_0,\beta^i_0,\dots,\beta^i_{k_i})\in\Delop^{k_i}$ for
$i=0,1,\dots,n$.

Properties 1)--4) imply that the operation $\lor:X\times X\to X$,
$x\lor y=x\op1y$ for all $x,y\in X$, is commutative, associative
and idempotent, thus $(X,\lor)$ is an upper semilattice with
a~partial order $x\le y\iff x\lor y=y$ for which $x \lor y$ is
a~pairwise supremum of $x$ and $y$. If $X$ is a compactum such that

6) for a neighborhood $U$ of any element $x\in X$ there is a
neighborhood $V$ of $x$, $V\subset U$, such that $y\op 1z\in V$ for
all $y,z\in V$;

then each point of $X$ has a local base consisting of
subsemilattices, and $(X,\lor)$ is a compact Lawson upper
semilattice~\cite{Laws69}. We will call a pair $(X,ic)$ of a
compactum $X$ with idempotent convex combination $ic$ that
satisfies the property 6) a~\emph{ $(\max,\min)$-idempotent convex
compactum}.

\begin{theo}\label{Mdn-alg}
Let $X$ be a compactum. There is a one-to-one correspondence
between continuous maps $\xi:\mdn X\to X$ such that the pair
$(X,\xi)$ is an~$\Mdn$-algebra, and continuous idempotent convex
combinations $ic:X\times I\times X\to X$ such that $(X,ic)$ is
a~$(\max,\min)$-idempotent convex compactum.

If for a continuous $ic:X\times I\times X\to X$ conditions 1)--5)
are valid, then 6) implies a stronger property~:

6+) for a neighborhood $U$ of any element $x\in X$ there is a
neighborhood $V$ of $x$, $V\subset U$, such that $y\op\alpha z\in
V$ for all $y,z\in V$, $\alpha\in I$.
\end{theo}

\begin{proof}
Let $(X,\xi)$ be an~$\Mdn$-algebra. Define the operation
$ic:X\times I\to X$ by the formula $ic(x,\alpha,y)=\xi(\delta_x\op
\alpha\delta_y)$. It is obvious that $\xi$ is well-defined,
continuous and satisfies 1), 4), 5). To prove 2), observe that by
the definition of an algebra for a monad we obtain
\begin{multline*}
(x\op\alpha y)\op \beta z=
\xi(\delta_{\xi(x\op\alpha y)}\op\beta\delta_z)=\\
\xi\circ \mdn\xi(\delta_{\delta_x\op\alpha\delta_y}\op\beta\delta_{\delta_z})=
\xi\circ \mudn
X(\delta_{\delta_x\op\alpha\delta_y}\op\beta\delta_{\delta_z})=\\
\xi(\delta_x\op\alpha\delta_y\op\beta\delta_z)=
\xi(\delta_x\op\beta\delta_z\op\alpha\delta_y)=
(x\op \beta z)\op \alpha y.
\end{multline*}
Proof of 3) is quite analogous. Thus the map $ic$ is an idempotent
convex combination of two points, and we consider idempotent convex
combinations of arbitrary finite number of points to be defined as
described above.

Let $U$ be a neighborhood of $x\in X$. By continuity of $\xi$ and
the equality $\xi(\delta_x)=x$ there is a neighborhood $\tilde
U\subset\mdn X$ of $\delta_x$ such that for all $c\in\tilde U$ we
have $\xi(c)\in U$. There also exists a neighborhood $\tilde V\ni
x$ such that for all $y_0,y_1,\dots,y_n\in \tilde V$,
$(\alpha_0,\alpha_1,\dots,\alpha_n)\in\Delop^n$ we have
$\alpha_0\delta_{y_0}\op\alpha_1\delta_{y_1}\op\dots\op
\alpha_n\delta{y_n}\in\tilde U$. It is straightforward to
verify that the set
$$
V=\{
\alpha_0y_0\op\alpha_1y_1\op\dots\op\alpha_ny_n
\mid
n\in\{0,1,\dots\}, (\alpha_0,\alpha_1,\dots,\alpha_n)\in\Delop^n,
y_0,y_1,\dots,y_n\in\tilde V
\}
$$
is a neighborhood of $x$ requested by 6+), which implies 6). Thus
it is proved that an $\Mdn$-algebra $(X,\xi)$ determines a
continuous operation $ic$ that satisfies conditions 1)--6).

Now assume that we are given a compactum $X$ and a continuous
operation $ic:X\times I\times X\to X$ that satisfies conditions
1)--6). Recall that $X$ with the operation $\lor:X\times X\to X$,
defined by the formula $x\lor y=x\op 1y$, is a compact Lawson upper
semilattice, therefore for all nonempty closed $F\subset X$ there
is $\sup F$ that depends on $F$ continuously w.r.t.\ Vietoris
topology~\cite{McWat}. Let $c\in\mdn X$ and $c(x_0)=1$ for some
$x_0\in X$. We put $\xi(c)=\sup\{x_0\op\alpha x\mid x\in
X,\alpha\le c(x)\}$. We will prove that $\xi:\mdn X\to X$ is well
defined (i.e. does not depend on the choice of $x_0$) and
continuous.

For each $x\in X$ let $gr(x)$ be the collection $(x\lor y)_{y\in
X}\in X^X$. Then the map of compacta $gr:X\to X^X$ is continuous
and injective, therefore is an embedding.

The equality
\begin{multline*}
\xi(c)\lor y=
\sup\{x_0\op\alpha x\mid  x\in X,\alpha\le c(x)\}\lor y=\\
\sup\{y\op 1x_0\op\alpha x\mid  x\in X,\alpha\le c(x)\}=
\sup\{(y\op 1x_0)\lor(y \op\alpha x)\mid  x\in X,\alpha\le
c(x)\}=\\
(y\op 1x_0)\lor\sup\{(y\op \alpha x\mid  x\in X,\alpha\le c(x)\}=
\sup\{(y\op \alpha x\mid  x\in X,\alpha\le c(x)\}.
\end{multline*}
holds for each $y\in X$, and the latter expression does not depend
on $x_0$. This implies that $gr(\xi(c))$ and thus $\xi(c)$ are
uniquely determined. Moreover, $pr_y\circ gr(\xi(c))$ is the
supremum of the image of the closed set $\{(x,\alpha)\mid x\in
X,\alpha\in I,\alpha\le c(x)\}\subset X\times I$ under the
continuous map that sends $(x,\alpha)$ to $y\op\alpha x\in X$.
Taking into account that this set (the \emph{hypograph} of the
function $c:X\to I$) depends on $c\in\mdn X$ continuously, we
obtain that the correspondence $c\mapsto gr(\xi(c))$ is continuous,
which implies continuity of $\xi:\mdn X\to X$.

To show that $(X,\xi)$ is an $\Mdn$-algebra, we again assume
$c(x_0)=1$ for a capacity $c\in \mdn X$. Then
\begin{gather*}
y\op\alpha\xi(c)=
y\op\alpha\sup\{x_0\op\beta x\mid  x\in X,\beta\le c(x)\}=\\
\sup\{y\op \alpha x_0\op(\alpha\om\beta) x\mid  x\in X,\beta\le
c(x)\}=\\
\sup\{(y\op \alpha x_0)\lor(y \op(\alpha \om\beta) x)\mid  x\in X,
\beta\le c(x)\}=\\
(y\op \alpha x_0)\lor \sup\{y \op(\alpha \om\beta) x\mid  x\in X,
\beta\le c(x)\}=\\
\sup\{y\op (\alpha\om\beta) x\mid  x\in X,\beta\le c(x)\}.
\end{gather*}
holds for each $y\in X$, $\alpha\in I$. The second equality sign
follows from an ``infinite distributive law'' $y\op\alpha\sup
F=\sup\{y\op\alpha x\mid x\in F\}$, with $F$ a nonempty subset of
$X$. This law is first proved for finite $F$ and then extended to
infinite case by continuity of lowest upper bounds.

It is obvious that $\xi(\delta_x)=x$ for a point $x\in X$, i.e.
$\xi\circ\etadn X=\uni{X}$. We choose a capacity $\CCC\in
\mdn^2X$ and compare $\xi\circ \mdn\xi(\CCC)$ and $\xi\circ\mudn
X(\CCC)$. For a point $y\in X$ we have
\begin{multline*}
y\lor(\xi\circ \mdn\xi(\CCC))=
\\
\sup\{(y\op \alpha x\mid  x\in X,\alpha\le \mdn\xi(\CCC)(x)\}=
\sup\{(y\op \alpha \xi(c)\mid  c\in\mdn X,\alpha\le \CCC(c)\}=
\\
\sup\{\sup\{y\op (\alpha\om\beta)x \mid  x\in X,\beta\le c(x)\},
c\in\mdn X,\alpha\le \CCC(c)\}=
\\
\sup\{(y\op \alpha x\mid  x\in X,c\in\mdn X,\alpha\le \min\{\CCC(c),c(x)\}\}=
y\lor (\xi\circ \mudn X(\CCC)).
\end{multline*}
This implies $\xi\circ \mdn\xi=\xi\circ\mudn X$, i.e. $(X,\xi)$ is
a $\Mdn$-algebra.

To prove that the correspondence ``$\Mdn$-algebra $\leftrightarrow$
idempotent convex combination that satisfies 1)--6)'' is
one-to-one, assume that for some continuous $ic:X\times I\times X$
satisfying 1)--6) there is a continuous map $\xi':\mdn X\to X$ such
that $(X,\xi')$ is a $\Mdn$-algebra and
$ic(x,\alpha,y)=\xi'(\delta_x\op\alpha\delta_y)$ for all $x,y\in
X$, $\alpha\in I$. Therefore $\xi'(\delta_x\op\alpha\delta_y)=
\xi(\delta_x\op\alpha\delta_y)$ for the constructed above map
$\xi$. Let $1\ge \alpha_1\ge\alpha_2\ge 0$, $x_0,x_1,x_2\in X$,
then
\begin{gather*}
\xi(\delta_{x_0}\op\alpha_1\delta_{x_1}\op\alpha_2\delta_{x_2})=
\xi\circ\mudn X(\delta_{\delta_{x_0}}
\op\alpha_1\delta_{\delta_{x_1}\op\alpha_2\delta_{x_2}})=\\
\xi\circ\mdn\xi(\delta_{\delta_{x_0}}
\op\alpha_1\delta_{\delta_{x_1}\op\alpha_2\delta_{x_2}})=
\xi(\delta_{x_0}\op\alpha_1\delta_{\xi(\delta_{x_1}\op\alpha_2\delta_{x_2})})=\\
\xi'(\delta_{x_0}\op\alpha_1\delta_{\xi'(\delta_{x_1}\op\alpha_2\delta_{x_2})})=
\dots=
\xi'(\delta_{x_0}\op\alpha_1\delta_{x_1}\op\alpha_2\delta_{x_2}).
\end{gather*}
By induction in a similar manner we prove that
$$
\xi(\delta_{x_0}\op\alpha_1\delta_{x_1}\op\alpha_2\delta_{x_2}\op
\dots\op\alpha_n\delta_{x_n})=
\xi'(\delta_{x_0}\op\alpha_1\delta_{x_1}\op\alpha_2\delta_{x_2}\op
\dots\op\alpha_n\delta_{x_n})
$$
for arbitrary integer $n\ge0$. By continuity we deduce that $\xi(c)=\xi'(c)$
for all $c\in\mdn X$.
\end{proof}

Let $ic:X\times I\times X\to X$ and $ic':X'\times I\times X'\to X'$
be idempotent convex combinations. We say that a~map $f:(X,ic)\to
(X',ic')$ is \emph{affine} if it preserves idempotent convex
combination, i.e. $f(ic(x,\alpha,y))=ic'(f(x),\alpha,f(y))$ for all
$x,y\in X$, $\alpha\in I$.

\begin{theo}\label{Mdn-mor}
Let $(X,\xi)$, $(X',\xi')$ be $\Mdn$-algebras, $ic:X\times I\times
X\to X$ and $ic':X'\times I\times X'\to X'$ be the respective
idempotent convex combinations. Then a continuous map $f:X\to Y$ is
a morphism of $\Mdn$-algebras $(X,\xi)\to(X',\xi')$ if and only if
$f:(X,ic)\to (X',ic')$ is affine.
\end{theo}

\begin{proof}
\textsl{Necessity.} Let $f:(X,\xi)\to(X',\xi')$ be a morphism of
$\Mdn$-algebras,$x,y\in X$, $\alpha\in I$. Then

$$
f(ic(x,\alpha,y))=f\circ\xi(\delta_x\lor\alpha\delta_y)=
\xi'\circ Mf(\delta_x\lor\alpha\delta_y)=
\xi'(\delta_{f(x)}\lor\alpha\delta_{f(y)})=ic'(f(x),\alpha,f(y)).
$$

\textsl{Sufficiency.} Let $f:(X,ic)\to(X',ic')$ be affine, then
$f(x\lor y)=f(x)\lor f(y)$ for all $x,y\in X$. Continuity of $f$
implies that $f$ preserves suprema of closed sets. For $c\in\mdn X$
we choose a point $x_0\in X$ such that $c(x_0)=1$, then $\mdn
f(x)(f(x_0))=1$. Therefore~:
\begin{multline*}
\xi'\circ \mdn f(c)=
\sup\{f(x_0)\op\alpha x'\mid x'\in
X',\alpha\le \mdn f(c)(x')\}=\\
\sup\{f(x_0)\op\alpha f(x)\mid x\in
X,\alpha\le c(x)\}=
\sup\{f(x_0\op\alpha x)\mid x\in
X,\alpha\le c(x)\}=\\
f(\sup\{x_0\op\alpha x\mid x\in
X,\alpha\le c(x)\})=
f\circ \xi(c),
\end{multline*}
and $f$ is a~morphism of $\Mdn$-algebras.

\end{proof}

\begin{rem}
It is easy to see that $(\max,\min)$-idempotent convex compacta and
their affine continuous maps constitute a category $\Convmm$ of
$(\max,\min)$-idempotent convex compacta that by the latter theorem
is monadic (=tripleable)~\cite{Swir74} over the category of
compacta.
\end{rem}

Convex compacta are usually defined as compact closed subsets of
locally convex topological vector spaces. To obtain a similar
description for $(\max,\min)$-idempotent convex compacta, we need
some extra definitions and facts. For an idempotent
semiring~\cite{DSIS03} $\CCS=(S,\op,\om,0,1)$ a (left idempotent)
$\CCS$-semimodule is a set $L$ with operations $\op:L\times L\to L$
and $\om:S\times L\to L$ such that for all $x,y,z\in L$,
$\alpha,\beta\in S$~:

1) $x\op y=y\op x$;

2) $(x\op y)\op z=x\op (y\op z)$;

3) there is an (obviously unique) element $\bar 0\in L$ such that
$x\op\bar 0=x$ for all $x$;

4) $\alpha\om(x\op y)=(\alpha\om x)\op(\alpha\om y)$,
$(\alpha\op\beta)\om x=(\alpha\om x)\op(\beta\om x)$;

5) $(\alpha\om\beta)\om x=\alpha\om(\beta\om x)$;

6) $1\om x=x$;

7) $0\om x=\bar 0$.

We adopt the usual convention and write $\alpha x$ instead of
$\alpha\om x$. Observe that these axioms imply $\alpha\bar 0=\bar
0$, $x\op x=x$. Informally speaking, an idempotent semimodule is a
vector space over an idempotent semiring.

If $\CCS=(I,\max,\min,0,1)$, we will talk about a
$(\max,\min)$-\emph{idempotent semimodule}. In this case we define
an operation $ic:L\times I\times L\to L$ by the formula
$ic(x,\alpha,y)=x\op(\alpha\om y)$ ($\op$ and $\om$ are from $L$).
It is easy to see that $ic$ satisfies 1)--5). The combination
$\alpha_0x_0\op\alpha_1x_0\op\dots\op\alpha_nx_n$ of points
$x_0,x_1,\dots,x_n$ is defined in an obvious way and coincides with
the described above operation if
$(\alpha_0,\alpha_1,\dots,\alpha_n)\in
\Delop^n$. A subset $A$ of a $(\max,\min)$-idempotent semimodule $L$
is called \emph{convex} if $x\op \alpha y\in A$ whenever $x,y\in
A$, $\alpha\in I$. A~convex subset $A\subset L$ contains all
idempotent convex combinations of its elements.

Let a $(\max,\min)$-idempotent semimodule $L$ be a compactum, the
operations $\op$ and $\om$ be continuous, and the topology on $L$
satisfy an additional condition~:

8) for a neighborhood $U$ of any element $x\in L$ there is a
neighborhood $V$ of $x$, $V\subset U$, such that $y\op z\in V$ for
all $y,z\in V$.

Then we call $(L,\op,\om)$ a~\emph{compact Lawson
$(\max,\min)$-idempotent semimodule}. By the above theorem $L$ is a
$\Mdn$-algebra, which implies

8+) for a neighborhood $U$ of any element $x\in L$ there is a
neighborhood $V$ of $x$, $V\subset U$, such that $y\op\alpha z\in
V$ for all $y,z\in V$, $\alpha\in I$.

Thus for every point of $L$ there is a local base that consists of
convex neighborhoods, and we say that $L$ is \emph{locally convex}.

The nature of a compactum $X$ with an idempotent convex
combination that satisfies 1)--6) is clarified by the following
\begin{theo}\label{Mdn-sm}
A pair of a compactum $X$ and a continuous map $ic:X\times I\times
X\to X$ is a $(\max,\min)$-idempotent convex compactum if and only
if $X$ is a closed convex subset of a compact Lawson
$(\max,\min)$-idempotent semimodule $(L,\op,\om)$ such that
$ic(x,\alpha,y)\equiv
\underset{\text{in }L}{\underbrace{x\op\alpha y}}$.
\end{theo}

\begin{proof}
Sufficiency is obvious. To prove necessity, assume that $X$ is a
compactum and a continuous map $ic:X\times I\times X\to X$
satisfies conditions 1)--6). We define an equivalence relation
``$\sim$'' on $X\times I$ as follows~: $(x_1,a_1)\sim (x_2,a_2)$ if
$y\op a_1x_1=y\op a_2x_2$ for all $y\in X$. This relation is closed
in $(X\times I)\times (X\times I)$, therefore the quotient space
$X\times I/{\sim}$, which we denote by $\bar X$, is a compact
Hausdorff space. We also denote by $[(x,a)]$ the equivalence class
of the pair $(x,a)$. The map $i:X\to\bar X$ that sends a point
$x\in X$ to $[(x,1)]$ is an embedding because $(x_1,1)\sim (x_2,1)$
is possible only if $x_1=x_2$.

We define operations $\om:I\times \bar X\to\bar X$ and $\op:\bar
X\times\bar X\to\bar X$ by the formulae
$\alpha\om[(x,a)]=[(x,\alpha\om a)]$ and
$$
[(x,a)]\op[(y,b)]=
\begin{cases}
[(x\op b y,a)],a\ge b,\\
[(y\op a y,b)],a\le b.
\end{cases}
$$

The element $\bar 0=[(x,0)]$ does not depend on $x$ and satisfies
3). Properties 5), 6), 7) are obvious. Verification that $\op$,
$\om$ are well defined, continuous and satisfy 1), 2), 4), 8), is
more convenient with a generalization of the mapping $gr:X\to X^X$
that was defined in the proof of the latter theorem. To avoid
introducing extra denotations, we denote by $gr(x,\alpha)$, where
$x\in X$, $\alpha\in I$, the collection $(t\op\alpha x)_{t\in X}$.
Then the map $gr:X\times I\to X^X$ is continuous (but, as can be
shown, not injective). It is obvious that $(x_1,\alpha_1)\sim
(x_2,\alpha_2)$ if and only if $gr(x_1,\alpha_1)=gr(x_2,\alpha_2)$,
thus we will identify the image of the map $gr$ with the quotient
space $\bar X=X\times I/{\sim}$, and $gr$ with the quotient map.

Let $\bar x,\bar y,\bar z$ be points in $\bar X$, and $\bar
x=gr(x,a)=(x_t)_{t\in X}$, $\bar y=gr(y,b)=(y_t)_{t\in X}$, $\bar
z=gr(z,c)=(z_t)_{t\in X}$. Observe that $x\op y=(x_t\lor
y_t)_{t\in}$, $\alpha\om \bar x=(t\op\alpha x_t)_{t\in X}$,
therefore $\bar x\op \bar y$ and $\alpha\om \bar x$ are uniquely
determined and continuous w.r.t.\ $\bar x$, $\bar y$ and $\alpha$,
$\bar x$ resp. Similar expressions can be written for $x\op z$ and
$y\op z$, and 1),2) are easily seen. Next, $\alpha\om \bar
x=(t\op\alpha x_t)_{t\in X}$, $\alpha\om \bar y=(t\op\alpha
y_t)_{t\in X}$, thus
$$
(\alpha\om \bar x)\op(\alpha\om \bar y)=
((t\op\alpha x_t)\lor(t\op\alpha x_t))_{t\in X}=
((t\op\alpha (x_t\lor\alpha x_t))_{t\in X}=
\alpha\om(\bar x\op\bar y).
$$
Similarly
$$
(\alpha\om \bar x)\op(\beta\om \bar x)=
((t\op\alpha x_t)\lor(t\op\beta x_t))_{t\in X}=
((t\op \alpha x_t\op\beta x_t))_{t\in X}=
(\alpha\op\beta)\om\bar x,
$$
and condition 4) holds.

Let $G\subset \bar X$ be a closed nonempty set, then $G=gr(F)$ for
some closed $F\subset X\times I$. There is $(x_0,a_0)\in F$ such
that $a_0=\max\{a\mid (x,a)\in F\}$. It is easy to show that $\sup
G$ in $\bar X$ is equal to $[(x',a_0)]$ where $x'=\sup\{x_0\op
ax\mid (x,a)\in F\}$, thus the upper semilattice $\bar X$ is
complete. It is also clearly seen that
$$
gr(x',a_0)=(\sup\{t\op ax\mid (x,a)\in F\})_{t\in X}=
(\sup\{x_t\mid (x_t)_{t\in X}\in G\})_{t\in X},
$$
therefore $\sup G$ depends on $G$ continuously w.r.t.\ Vietoris
topology. It is a statement equivalent to 8)~\cite{McWat}.
\end{proof}

As triples $\Mdn$ and $\Mup$ are isomorphic through a~natural
transformation $\kappa$ defined in~\cite{NHl}, and the map $I\to I$
that sends each $t$ to $1-t$ is an isomorphism of the idempotent
semirings $(I,\op,\om,0,1)$ and $(I,\om,\op,1,0)$, by duality we
immediately can state an analogue of Theorem~\ref{Mdn-alg}. Its
proof can be obtained by replacing $\mdn$ by $\mup$, $\om$ by
$\op$, $1$ by $0$, upper semilattices by lower ones, $\lor$ by
$\land$, $\sup$ by $\inf$, $\Delop$ by the~\emph{(idempotent)
$n$-dimensional $\om$-simplex}
$$
\Delom^n=\{(\alpha_0,\alpha_1,\dots,\alpha_n)\in
I^{n+1}\mid \alpha_0\om\alpha_1\om\dots\om\alpha_n=0\}
$$
and vice versa, where it is necessary. Thus we define
\emph{dual idempotent convex combinations} and
$(\min,\max)$-\emph{idempotent convex compacta} that are precisely
$\Mup$-algebras. We omit obvious details. Observe that for a given
$\Mup$-algebra $(X,\xi)$ the respective dual idempotent convex
combination $ci:X\times I\times X\to X$ is determined by the
equality $ci(x,\alpha,y)=\xi(\delta_x\land(\alpha\lor\delta_y))$.
Conversely, the value $\xi(c)$ for a capacity $c\in\mdn X$
(assuming that $c(X\setminus
\{x_0\})=0$) is equal to $\xi(c)=\inf\{ci(x_0,\alpha,x)\mid x\in
X,\alpha\ge c(X\setminus\{x)\}$.

It is easy also to formulate analogues of
Theorems~\ref{Mdn-mor},\ref{Mdn-sm}.

\section{Algebras for the capacity monad}

In the sequel a~\emph{$(\min,\max)$-idempotent biconvex compactum}
is a compactum $X$ with four operations $\bp:X\times X\to X$,
$\om:I\times X\to X$, $\bm:X\times X\to X$, $\op:I\times X\to X$
such that $(X,\bp,\bm)$ is a Lawson lattice, $(X,\bp,\om)$ is an
$(I,\op,\om)$-semimodule, $(X,\bm,\op)$ is an
$(I,\om,\op)$-semimodule, the associative laws $(\alpha\op
x)\bp y=\alpha\op(x\bp y)$, $(\alpha\om
x)\bm y=\alpha\om(x\bm y)$ and the distributive laws
$\alpha\om(\beta\op x)=(\alpha\om\beta)\op(\alpha\om x)$,
$\alpha\op(\beta\om x)=(\alpha\op\beta)\om(\alpha\op x)$
are valid for all $x,y\in X$, $\alpha,\beta\in I$.

\begin{theo}\label{M-alg}
Let $X$ be a compactum. There is a one-to-one correspondence
between~:

1) continuous maps $\xi:MX\to X$ such that the pair $(X,\xi)$
is an~$\BBM$-algebra;

2) quadruples $(\bp,\om,\bm,\op)$ of
continuous operations $\bp:X\times X\to X$, $\om:I\times X\to X$,
$\bm:X\times X\to X$, $\op:I\times X\to X$ such that
$(X,\bp,\om,\bm,\op)$ is a~$(\max,\min)$-idempotent biconvex
compactum;

3) quadruples $(\bp,\om,p,m)$ of continuous maps $\bp:X\times X\to X$,
$\bm:X\times X\to X$, $p,m:I\to X$ such that
\begin{description}
\item{a)} $(X,\bp,\bm)$ is a Lawson lattice;
\item{b)} $p:(I,\op)\to(X,\bp)$ is a morphism of upper
semilattices that preserves a top element;
\item{c)} $m:(I,\om)\to(X,\bm)$ is a morphism of lower
semilattices that preserves a bottom element;
\item{d)} for all $\alpha,\beta\in I$ we have
$m(\alpha)\om p(\beta)=p(\alpha\om\beta)$,
$m(\alpha)\op p(\beta)=m(\alpha\op\beta)$.
\end{description}

 In the case 2) the following property of \emph{local
biconvexity} holds~: for a neighborhood $U$ of any element $x\in X$
there is a neighborhood $V$ of $x$, $V\subset U$, such that
$y\bp(\alpha\bm z)\in V$, $y\bm(\alpha\bp z)\in V$ for all $y,z\in
V$, $\alpha\in I$.
\end{theo}

\begin{proof}

1)$\to$3). Let $(X,\xi)$ be an~$\BBM$-algebra. We use the fact that
$\BBG$ is a~submonad of the capacity monad $\BBM$. The components
of an embedding $i_G:\BBG\hra\BBM$ are of the form
$$
i_GX(\CCA)(F)=
\begin{cases}
1, \text{ if }F\in \CCA,\\
0\text{ --- otherwise}.
\end{cases}
$$
Therefore $(X,\xi\circ i_GX)$ is a~$\BBG$-algebra.
Theorem~2~\cite{RadG90} states that for a $\BBG$-algebra
$(X,\theta)$ the operations $\bp:X\times X\to X$, $\om:I\times X\to
X$ defined by the formulae $x\bp
y=\theta(\eta_GX(x)\cap\eta_GX(y))$ and $x\bm
y=\theta(\eta_GX(x)\cup\eta_GX(y))$ are such that $(X,\bp,\bm)$ is
a Lawson lattice. We apply this theorem to $\theta=\xi\circ i_GX$
and obtain that $X$ with the operations $x\bp
y=\xi(\delta_x\lor\delta_y)$ and $x\bm
y=\xi(\delta_x\land\delta_y)$ is a Lawson lattice. We denote by
$\bar 0$ and $\bar 1$ its least and greatest elements. Now we put
$p(\alpha)=\xi(\delta_{\bar 0}\lor
\alpha\om\delta_{\bar 1})$, $m(\alpha)=\xi(\delta_{\bar 1}\land
\alpha\op\delta_{\bar 0})$. It is obvious that $p,m$ are continuous
and $p(1)=\bar 0\bp\bar 1=\bar 1$, $m(0)=\bar 1\bm\bar 0=\bar 0$.
Next, for all $\alpha,\beta\in I$~:
\begin{gather*}
p(\alpha\op \beta)=
\xi(\delta_{\bar 0}\lor (\alpha\op\beta)\om\delta_{\bar 1})=
\xi\circ \mu X(
\delta_{\delta{\bar 0}\lor \alpha\om\delta_{\bar 1}}
\lor
\delta_{\delta{\bar 0}\lor \beta\om\delta_{\bar 1}})=\\
\xi\circ M\xi(
\delta_{\delta{\bar 0}\lor \alpha\om\delta_{\bar 1}}
\lor
\delta_{\delta{\bar 0}\lor \beta\om\delta_{\bar 1}})=
\xi(
\delta_{p(\alpha)}\lor \delta_{p(\beta)})=p(\alpha)\bp p(\beta).
\end{gather*}
Similarly $m(\alpha\om \beta)=m(\alpha)\bm m(\beta)$ for all
$\alpha,\beta\in I$. We also have
\begin{gather*}
m(\alpha)\om p(\beta)=
\xi(
\delta_{\xi(\delta_{\bar 1}\land \alpha\op\delta_{\bar 0})}
\land
\delta_{\xi(\delta_{\bar 0}\lor \beta\om\delta_{\bar 1})}
)=
\xi\circ M\xi(
\delta_{\delta_{\bar 1}\land \alpha\op\delta_{\bar 0}}
\land
\delta_{\delta_{\bar 0}\lor \beta\om\delta_{\bar 1}}
)=\\
\xi\circ \mu X(
\delta_{\delta_{\bar 1}\land \alpha\op\delta_{\bar 0}}
\land
\delta_{\delta_{\bar 0}\lor \beta\om\delta_{\bar 1}}
)=
\xi(
\delta_{\bar 0}\lor (\alpha\om\beta)\om\delta_{\bar 1}
)=p(\alpha\om\beta),
\end{gather*}
as well as $m(\alpha)\op p(\beta)=m(\alpha\op\beta)$.

3)$\to$2). It is sufficient to put $\alpha\om x=m(\alpha)\bm x$,
$\alpha\op x=p(\alpha)\bp x$, and it is clear that all conditions
of 2) are satisfied due to the commutative, associative and
distributive laws in $(X,\bp,\bm)$.

Observe also that, if $m,p$ are determined by an~$\BBM$-algebra $(X,\xi)$
as described above, then
\begin{gather*}
\alpha\om x=
\xi(
\delta_{\xi(\delta_{\bar 1}\land\alpha\op\delta_{\bar 0})}
\land \delta_x)=
\xi(
\delta_{\xi(\delta_{\bar 1}\land\alpha\op\delta_{\bar 0})}
\land \delta_x)=
\xi\circ M\xi(
\delta_{\delta_{\bar 1}\land\alpha\op\delta_{\bar 0}}
\land \delta_{\delta_x}
)=\\
\xi\circ \mu X(
\delta_{\delta_{\bar 1}\land\alpha\op\delta_{\bar 0}}
\land \delta_{\delta_x}
)=
\xi(
\delta_{\bar 1}\land\alpha\op\delta_{\bar 0}\land \delta_x
)=
\xi\circ \mu X(
\delta_{\delta_{x}\land\alpha\op\delta_{\bar 0}}
\land \delta_{\delta_{\bar 1}}
)=\\
\xi\circ M\xi(
\delta_{\delta_{x}\land\alpha\op\delta_{\bar 0}}
\land \delta_{\delta_{\bar 1}}
)=
\xi(\delta_{x}\land\alpha\op\delta_{\bar 0})\bm\bar 1=
\xi(\delta_{x}\land\alpha\op\delta_{\bar 0}),
\end{gather*}
and similarly $\alpha\op x=\xi(\delta_x\lor \alpha\om\delta_{\bar
1})$ for all $x\in X$, $\alpha\in I$. In the same manner we can
show that $x\bp(\alpha\om
y)=\xi(\delta_{x}\lor\alpha\om\delta_{y})$, $x\bm(\alpha\op
y)=\xi(\delta_{x}\land\alpha\op\delta_{y})$ for all $x,y\in X$,
$\alpha\in I$. These formulae are the same that were used to define
idempotent semiconvex combinations and dual idempotent semiconvex
combinations in the proofs of Theorem~\ref{Mdn-alg} and the dual
theorem.

2)$\to$1).

Now let $(X,\bp,\om,\bm,\op)$ be a~$(\min,\max)$-idempotent
biconvex compactum. If $ic(x,\alpha,y)=x\bp(\alpha\om y)$,
$ci(x,\alpha,y)=x\bm(\alpha\op y)$, then it is obvious that
$(X,ic)$ is a $(\min,\max)$-idempotent convex compactum and
$(X,ci)$ is a $(\max,\min)$-idempotent convex compactum. Thus by
Theorem~\ref{Mdn-alg} and the dual theorem, if mappings $\xidn:\mdn
X\to X$ and $\xiup:\mup X\to X$ are defined by the formulae
$$
\xidn(c)=\sup\{x_0\bp(\alpha\om x)\mid x\in X,\alpha\le c(x)\},
\quad c\in \mdn X,x_0\in X, c(x_0)=1,
$$
and
$$
\xiup(c)=\inf\{x_0\bm(\alpha\op x)\mid x\in X,\alpha\ge c(X\setminus\{x\})\},
\quad c\in \mup X,x_0\in X, c(X\setminus\{x_0\})=0,
$$
then the pairs $(X,\xidn)$ and $(X,\xiup)$ are resp.\
an~$\Mdn$-algebra and an~$\Mup$-algebra. In our case we can define
$\xidn,\xiup$ by simpler but equivalent formulae (the second ``$=$'' sign
in each equality is due to complete distributivity of a compact Lawson lattice)~:
$$
\xidn(c)=
\sup\{c(x)\om x\mid x\in X\}
=\inf\{c(X\setminus A)\op\sup A\mid A\subcl X\},
\quad c\in \mdn X,
$$
and
$$
\xiup(c)=
\inf\{c(X\setminus\{x\})\op x\mid x\in X\}
=\sup\{c(X\setminus A)\om\inf A\mid A\subcl X\},
\quad c\in \mup X.
$$

If $\xi,\xi':MX\to X$ are continuous maps such that the~pairs
$(X,\xi)$, $(X,\xi')$ are $\BBM$-algebras and $\xi|_{\mdn
X}=\xi'|_{\mdn X}=\xidn$, $\xi|_{\mup X}=\xi'|_{\mup X}=\xiup$,
then the two following diagrams have to be commutative (we omit
explicit notations for restrictions)~:
\begin{equation*}
\begin{gathered}
\xymatrix{
\mdn\mup X
\ar[r]^{\mu X}
\ar[d]_{\mdn\xiup}
&
MX
\ar[d]^{\xi}
\\
\mdn
\ar[r]_{\xidn}
&
X
}
\end{gathered}
\quad\text{(*)}
\qquad
\begin{gathered}
\xymatrix{
\mup\mdn X
\ar[r]^{\mu X}
\ar[d]_{\mup\xidn}
&
MX
\ar[d]^{\xi'}
\\
\mup
\ar[r]_{\xiup}
&
X
}
\end{gathered}
\quad\text{(**)}
\end{equation*}

We show that if $\CCC,\CCC'\in\mup\mdn X$ are such that $\mu
X(\CCC)=\mu X(\CCC)$, then
$\xidn\circ\mdn\xiup(\CCC)=\xidn\circ\mdn\xiup(\CCC')$. Observe
that $\mu X(\CCC)=\mu X(\CCC)$ implies that for all $A\subcl X$ and
$\alpha\in I$ the existence of $c\in\mup X$ such that
$\CCC(c)\ge\alpha$ and $c(A)\ge \alpha$ is equivalent to the
existence of $c'\in\mup X$ such that $\CCC'(c')\ge\alpha$ and
$c'(A)\ge \alpha$. It is also obvious that the same statement is
valid for any open $A\subset X$. Thus~:
\begin{gather*}
\xidn\circ\mdn\xiup(\CCC)=
\sup\{\mdn\xiup(\CCC)(x)\om x\mid x\in X\}=\\
\sup\{\CCC(c)\om\xiup(c)\mid c\in\mup X\}=\\
\sup\{\CCC(c)\om
\sup\{c(X\setminus A)\om\inf A\mid A\subcl X\}
\mid c\in\mup X\}=\\
\sup\{\CCC(c)\om c(X\setminus A)\om\inf A
\mid A\subcl X, c\in\mup X\}=\\
\sup\{\alpha\om\inf A
\mid A\subcl X,c\in\mup X,\alpha\in I,
\alpha\le\CCC(c),\alpha\le c(X\setminus A)\}=\\
\sup\{\alpha\om\inf A
\mid A\subcl X,c'\in\mup X,\alpha\in I,
\alpha\le\CCC'(c'),\alpha\le c'(X\setminus A)\}=\\
\dots =\xidn\circ\mdn\xiup(\CCC').
\end{gather*}

An obvious dual statement is also valid. Taking into account that
by Theorem~8~\cite{NHl} for a compactum $X$ the equality
$\mu(\mup\mdn X)=\mu(\mdn\mup X)=MX$ is valid, and
$\mu_X|_{\mup\mdn X}:\mup\mdn X\to MX$ and $\mu_X|_{\mdn\mup
X}:\mdn\mup X\to MX$ are quotient maps as continuous surjective
maps of compacta, we obtain that the diagrams (*) and (**) uniquely
determine continuous maps $\xi,\xi':MX\to X$.

In the diagram
$$
\xymatrix{
\mdn^2\mup X
\ar[rr]^{\mudn\mup X}
\ar[rd]^{\mdn^2\xiup}
\ar[dd]_{\mdn\mu X}
& &
\mdn\mup X
\ar[rd]^{\mdn\xiup}
\ar[dd]_>>>>>{\mu X}
\\
&
\mdn^2X
\ar[rr]^<<<<<<<{\mudn X}
\ar[dd]_<<<<{\mdn\xidn}
& &
\mdn X
\ar[dd]^{\xidn}
\\
\mdn MX
\ar[rr]^>>>>>>>{\mu X}
\ar[rd]_{\mdn\xi}
& &
MX
\ar[rd]^{\xi}
\\
&
\mdn X
\ar[rr]_{\xidn}
& &
X
}
$$
the top square and the side squares are commutative, and the
leftmost vertical arrow is epimorphic, therefore the bottom square
commutes as well. Using also dual arguments, we show that the two
following diagrams are commutative~:
\begin{equation*}
\begin{gathered}
\xymatrix{
\mdn M X
\ar[r]^{\mu X}
\ar[d]_{\mdn\xi}
&
MX
\ar[d]^{\xi}
\\
\mdn X
\ar[r]_{\xidn}
&
X
}
\end{gathered}
\qquad
\begin{gathered}
\xymatrix{
\mup M X
\ar[r]^{\mu X}
\ar[d]_{\mup\xidn'}
&
MX
\ar[d]^{\xi'}
\\
\mup X
\ar[r]_{\xiup}
&
X
}
\end{gathered}
\end{equation*}

We apply the functor $\mup$ to the left diagram and combine it with
(**)~:
$$
\xymatrix{
&
M^2X
\ar[rd]^{M\xi}
\\
\mup\mdn MX
\ar[ur]^{\mu MX}
\ar[r]^{\mup\mdn \xi}
\ar[d]^{\mup\mu X}
\ar@/_5ex/[dd]_{\mu MX}
&
\mup\mdn X
\ar[r]^{\mu X}
\ar[d]^{\mup\xidn}
&
MX
\ar[d]^{\xi'}
\\
\mup MX
\ar[r]^{\mup\xi}
\ar[rd]^{\mu X}
&
\mup X
\ar[r]^{\xiup}
&
X
\\
M^2X\ar[r]_{\mu X}
&
MX
\ar[ur]_{\xi'}
}
$$
The restriction of $\mu MX$ to $\mup\mdn MX$ is an epimorphism,
therefore the commutativity of the outer contour imply that the
left of the two following diagrams commutes. The right diagram is
commutative by dual arguments.
\begin{equation*}
\begin{gathered}
\xymatrix{
M^2X
\ar[r]^{\mu X}
\ar[d]_{M\xi}
&
MX
\ar[d]^{\xi'}
\\
MX
\ar[r]_{\xi'}
&
X
}
\end{gathered}
\qquad
\begin{gathered}
\xymatrix{
M^2X
\ar[r]^{\mu X}
\ar[d]_{M\xi'}
&
MX
\ar[d]^{\xi}
\\
MX
\ar[r]_{\xi}
&
X
}
\end{gathered}
\end{equation*}

Therefore in the diagram
$$
\xymatrix{
M^3X
\ar[rr]^{M\mu X}
\ar[rd]^{\mu MX}
\ar[dd]_{M^2\xi}
& &
M^2X
\ar[rd]^{\mu X}
\ar[dd]_>>>>>{M\xi'}
\\
&
M^2X
\ar[rr]^<<<<<<<{\mu X}
\ar[dd]_<<<<{M\xi}
& &
MX
\ar[dd]^{\xi}
\\
M^2X
\ar[rr]^>>>>>>>{M\xi'}
\ar[rd]_{\mu X}
& &
MX
\ar[rd]^{\xi}
\\
&
MX
\ar[rr]_{\xi}
& &
X
}
$$
the front square is commutative, which is the ``harder part'' of
the definition of $\BBM$-algebra. A proof of the ``easier part''
$\xi\circ\eta X=\uni{X}$ is straightforward. Thus $(X,\xi)$ is
a~\emph{unique} $\BBM$-algebra such that $x\bp(\alpha\om
y)=\xi(\delta_x\lor (\alpha\om\delta_y))$ and $x\bm(\alpha\op
y)=\xi(\delta_x\land(\alpha\op\delta_y))$ for all $x,y\in X$,
$\alpha\in I$. As a by-product we obtain that $\xi=\xi'$, i.e.
defininions of $\xi$ by the diagrams (*) and (**) are equivalent.

To prove local biconvexity, for a given neighborhood $U$ of a point
$x$ by continuity of $\xi$ and the equality $\xi(\delta_x)=x$ we
choose a neighborhood $\tilde U\subset\mdn X$ of $\delta_x$ such
that for all $c\in\tilde U$ we have $\xi(c)\in U$. There exists
a~neighborhood $\hat U\ni x$ such that for all
$x_0,x_1,\dots,x_n\in
\hat U$, $(\alpha_0,\alpha_1,\dots,\alpha_n)\in\Delop^n$ we have
$\alpha_0\delta_{x_0}\lor\alpha_1\delta_{x_1}\lor\dots\lor
\alpha_n\delta{x_n}\in\tilde U$. Now we choose a~neighborhood $\tilde{\tilde
U}\subset\mup X$ of $\delta_x$ such that for all $c\in\tilde U$ we
have $\xi(c)\in \hat U$. There is also a~neighborhood $\hat{\hat
U}\ni x$ such that $y_0,y_1,\dots,y_n\in \hat{\hat U}$,
$(\alpha_0,\alpha_1,\dots,\alpha_n)\in\Delom^n$ imply
$\alpha_0\delta_{y_0}\land\alpha_1\delta_{y_1}\land\dots\land
\alpha_n\delta{y_n}\in \tilde{\tilde U}$. Now we put
\begin{multline*}
\tilde V=\{
(\alpha_\op 0y_0)\bm(\alpha_1\op y_1)\bm\dots\bm(\alpha_n\op y_n)
\mid
\\
n\in\{0,1,\dots\}, (\alpha_0,\alpha_1,\dots,\alpha_n)\in\Delom^n,
y_0,y_1,\dots,y_n\in\hat{\hat U}
\},
\end{multline*}
and the set
\begin{multline*}
V=\{
(\alpha_0\om y_0)\bp(\alpha_1\om y_1)\bp\dots\bp(\alpha_n\om y_n)
\mid
\\
n\in\{0,1,\dots\}, (\alpha_0,\alpha_1,\dots,\alpha_n)\in\Delop^n,
x_0,x_1,\dots,x_n\in\tilde V
\}
\end{multline*}
is a neighborhood of $x$ requested by local bicommutativity.
\end{proof}

For $(\max,\min)$-idempotent biconvex compacta
$(X,\bp,\om,\bm,\op)$ and $(X',\bp,\om,\bm,\op)$ we say that a~map
$f:X\to X'$ is \emph{biaffine} if it preserves idempotent convex
combination and the dual idempotent convex combination, i.e.
$f(x\bp(\alpha\om y))=f(x)\bp(\alpha\om f(y))$, $f(x\bm(\alpha\op
y))=f(x)\bm(\alpha\op f(y))$ whenever $x,y\in X$, $\alpha\in I$.

\begin{theo}
Let $(X,\xi)$, $(X',\xi')$ be $\BBM$-algebras and quadruples
$(\bp,\om,\bm,\op)$ of continuous operations be determined on $X$
and $X'$ by $\xi$ and $\xi'$ resp.\ (in the sense of
Theorem~\ref{M-alg}).Then a continuous map $f:X\to Y$ is a morphism
of $\Mdn$-algebras $(X,\xi)\to(X',\xi')$ if and only if
$(X,\bp,\om,\bm,\op)\to(X',\bp,\om,\bm,\op)$ is biaffine.
\end{theo}

\begin{proof}
\textsl{Necessity.}
Let $f$ be a morphism of algebras. It was shown in the proof of the
previous theorem that the idempotent convex combination and the
dual idempotent convex combination of points $x,y\in X$ are
determined by the formulae $x\bp(\alpha\om
y)=\xi(\delta_{x}\lor\alpha\om\delta_{y})$, $x\bm(\alpha\op
y)=\xi(\delta_{x}\land\alpha\op\delta_{y})$ (in $X'$ the same but
$\xi$ replaced with $\xi'$). Then we follow the line of the proof
of Theorem~\ref{Mdn-mor}.

\textsl{Sufficiency.}
Let $f$ be biaffine. Then by Theorem~\ref{Mdn-mor} and a dual
theorem $f$ is a morphism of $\mdn$-algebras $(X,\xi|_{\mdn X})\to
(X',\xi'|_{\mdn X'})$ and a morphism of $\mup$-algebras
$(X,\xi|_{\mup X})\to (X',\xi'|_{\mup X'})$, i.e. the diagrams
\begin{equation*}
\begin{gathered}
\xymatrix{
\mdn X
\ar[r]^{\mdn f}
\ar[d]_{\xi|_{\mdn X}}
&
\mdn X'
\ar[d]^{\xi'|_{\mdn X'}}
\\
X
\ar[r]_{f}
&
X'
}
\end{gathered}
\qquad
\begin{gathered}
\xymatrix{
\mup X
\ar[r]^{\mup f}
\ar[d]_{\xi|_{\mup X}}
&
\mup X'
\ar[d]^{\xi'|_{\mup X'}}
\\
X
\ar[r]_{f}
&
X'
}
\end{gathered}
\end{equation*}
are commutative. Therefore the top face and the side faces of the
diagram
$$
\xymatrix{
\mdn\mup X
\ar[rr]^{\mdn\mup f}
\ar[rd]^{\mdn(\xi|_{\mup X})}
\ar[dd]_{\mu X}
& &
\mdn\mup X'
\ar[rd]^{\mdn(\xi'|_{\mup X'})'}
\ar[dd]_>>>>>{\mu X'}
\\
&
\mdn MX
\ar[rr]^<<<<<<<{\mdn f}
\ar[dd]_<<<<{\xi|_{\mdn X}}
& &
\mdn MX'
\ar[dd]^{\xi'|_{\mdn X'}}
\\
MX
\ar[rr]^>>>>>>>{Mf}
\ar[rd]_{\xi}
& &
MX'
\ar[rd]^{\xi'}
\\
&
X
\ar[rr]_{f}
& &
X'
}
$$
commute. The leftmost arrow $\mu X:\mdn\mup X\to MX$ is an
epimorphism, thus the bottom face commutes as well, i.e. $f$ is a
morphism of $\BBM$-algebras.
\end{proof}

\begin{rem}
The latter theorem implies that the category $\BConvmm$ of
$(\max,\min)$-idempotent biconvex compacta and their continuous
biaffine maps is monadic over the category of compacta.
\end{rem}

\begin{rem}
Note that a biaffine map $f:(X,\bp,\om,\bm,\op)\to
(X',\bp,\om,\bm,\op)$ not necessarily preserves operations $\op$
and $\om$ (although it preserves $\bp$ and $\bm$). E.g., let
$X=X'=I$, $\op=\bp=\max$, $\om=\bm=\min$, $f(x)=\max\{x,\frac12\}$.
Then $f$ is biaffine, but $f(0\om 1)=\frac12\ne 0\om f(1)=0$. It is
easy to show that a biaffine continuous map
$f:(X,\bp,\om,\bm,\op)\to (X',\bp,\om,\bm,\op)$ preserves $\om$ iff
it preserves a bottom element, and it preserves $\op$ iff it
preserves a top element.
\end{rem}

We present an example of $(\max,\min)$-idempotent biconvex
compacta. Let $A$ be a set and for each $a\in A$ a non-decreasing
surjective map $\phi_a:I\to I$ is fixed. For $x,y\in I^A$,
$x=(x_a)_{a\in A}$, $y=(y_a)_{a\in A}$, $\alpha\in I$ we put $x\bp
y=(\max\{x_a,y_a\})_{a\in A}$, $x\bm y= (\min\{x_a,y_a\})_{a\in
A}$, $\alpha\om x=(\min\{\phi_a(\alpha),x_a\})_{a\in A}$,
$\alpha\op x=(\max\{\phi_a(\alpha),x_a\})_{a\in A}$. Then
$(X,\bp,\om,\bm,\op)$ obviously satisfies the definition. In
communication with M.~Zarichnyi a question arose~:

\begin{que}
Does every $(\max,\min)$-idempotent biconvex compactum
biaffinely
embed into some $I^A$ with the defined above operations?
\end{que}

Provided the answer is positive, any biconvex map
$f:(X,\bp,\om,\bm,\op)\to(X',\bp,\om,\bm,\op)$ algebraically (with
preservation of idempotent and dual idempotent convex combinations)
and topologically embeds into a biconvex map that is a projection
of some $I^A$ onto $I^B$, $B\subset A$ (operations on $I^A$ and
$I^B$ are defined as above).

\end{document}